\documentclass[5p]{elsarticle}
\usepackage{lineno,hyperref}

\usepackage{bm}
\usepackage{amsmath}
\usepackage{amsthm}
\usepackage{amsfonts}
\usepackage{graphics}
\usepackage{color}
\usepackage{framed}
\usepackage{subfig}
\newtheorem{theorem}{Theorem}
\newtheorem{lemma}{Lemma}

\newtheorem{corollary}{Corollary}
\newtheorem{assum}{Assumption}

\newtheorem{remark}{Remark} 
\newcommand{\pa}{\partial}
\newcommand{\ba}{\begin{align}}
\newcommand{\ea}{\end{align}}
\newcommand{\fr}{\frac}
\newcommand{\gm}{\gamma}
\newcommand{\ep}{\varepsilon}

\newcommand{\alp}{\alpha}

\usepackage{mathptmx}
\allowdisplaybreaks

\journal{Journal of \LaTeX\ Templates}

\begin{document}

\begin{frontmatter}

\title{\textbf{Single-Boundary Control of the Two-Phase Stefan System}}

\author{Shumon Koga}\corref{my}\cortext[my]{Corresponding author}\ead{skoga@eng.ucsd.edu}               
\author{Miroslav Krstic}\ead{krstic@ucsd.edu}  

\address{Department of Mechanical and Aerospace Engineering, University of California, San Diego, La Jolla, CA 92093-0411 USA}

\begin{abstract}
This paper presents the control design of the two-phase Stefan problem. The two-phase Stefan problem is a representative model of liquid-solid phase transition by describing the time evolutions of the temperature profile which is divided by subdomains of liquid and solid phases as the liquid-solid moving interface position. The mathematical formulation is given by two diffusion partial differential equations (PDEs) defined on a time-varying spatial domain described by an ordinary differential equation (ODE) driven by the Neumann boundary values of both PDE states, resulting in a nonlinear coupled PDE-ODE-PDE system. We design a state feedback control law by means of energy-shaping to stabilize the interface position to a desired setpoint by using single boundary heat input. We prove that the closed-loop system under the control law ensures some conditions for model validity and the global exponential stability estimate is shown in $L_2$ norm. Furthermore, the robustness of the closed-loop stability with respect to the uncertainties of the physical parameters is shown. Numerical simulation is provided to illustrate the good performance of the proposed control law in comparison to the control design for the one-phase Stefan problem.  \end{abstract}


\end{frontmatter}


\section{Introduction} 
Liquid-solid phase transitions are physical phenomena which appear in various kinds of science and engineering processes. Representative applications include sea-ice melting and freezing~\cite{koga2019arctic}, continuous casting of steel \cite{petrus2012}, cancer treatment by cryosurgeries \cite{Rabin1998}, additive manufacturing for materials of both polymer \cite{koga2018polymer} and metal \cite{chung2004}, crystal growth~\cite{conrad_90}, lithium-ion batteries \cite{koga2017battery}, and thermal energy storage systems \cite{zalba03}. Physically, these processes are described by a temperature profile along a liquid-solid material, where the dynamics of the liquid-solid interface is influenced by the heat flux induced by melting or solidification. A mathematical model of such a physical process is called the Stefan problem\cite{Gupta03}, which is formulated by a diffusion PDE defined on a time-varying spatial domain. The domain's length dynamics is described by an ODE dependent on the Neumann boundary value of the PDE state. Apart from the thermodynamical model, the Stefan problem has been employed to model several chemical, electrical, social, and financial dynamics such as tumor growth process \cite{Friedman1999}, domain walls in ferroelectric thin films \cite{mcgilly2015}, spreading of invasive species in ecology \cite{Du2010speading}, information diffusion on social networks \cite{Lei2013}, and optimal exercise boundary of the American put option on a zero dividend asset \cite{Chen2008}.

The mathematical and numerical analysis of the Stefan problem has been widely covered in literature. The existence and uniqueness of the classical solution of the two phase Stefan problem was proven in \cite{Cannon71temp,Cannon71flux} with the temperature boundary conditions and the flux boundary conditions, respectively. Several numerical methods to solve the Stefan problem was investigated, see \cite{kutluay97} for instance. The comparison of the numerical methods was studied in \cite{Javierre06}. However, the control related problems have been considered relatively fewer. 

In \cite{Hinze07}, an optimal control approach for the solidification process described by the two-phase Stefan problem has been developed via adjoint method to track the phase interface to a prescribed desired motion. While their results showed the novelties by implementing the method for the two-dimensional system, the iterative method utilized for the optimization problem is computationally expensive and not robust to the unknown disturbances and physical parameters. For control objectives, infinite-dimensional approaches have been used for stabilization of  the temperature profile and the moving interface of a 1D one-phase Stefan problem, such as enthalpy-based feedback~\cite{petrus2012} and geometric control~\cite{maidi2014}.  These works designed control laws ensuring the asymptotical stability of the closed-loop system in the ${L}_2$ norm. However, the results in \cite{maidi2014} are established based on the assumptions on the liquid temperature being greater than the melting temperature, which must be ensured by showing the positivity of the boundary heat input.  

Recently, boundary feedback controllers for the Stefan problem have been designed via a ``backstepping transformation" \cite{krstic2008boundary,andrew2004} which has been used for many other classes of infinite-dimensional systems. For instance, \cite{Shumon16} designed a state feedback control law by introducing a nonlinear backstepping transformation for moving boundary PDE, which achieved the exponentially stabilization of the closed-loop system in the ${\mathcal H}_1$ norm without imposing any {\em a priori} assumption. Based on the technique, \cite{Shumon16CDC} designed an observer-based output feedback control law for the Stefan problem, \cite{Shumon19journal} extended the results in \cite{Shumon16, Shumon16CDC} by studying the robustness with respect to the physical parameters and developed an analogous design with Dirichlet boundary actuation, \cite{Shumon17ACC} designed a state feedback control for the Stefan problem under the material's convection, \cite{koga_2019delay} developed a control design with time-delay in the actuator and proved a delay-robustness, and \cite{koga2019iss} investigated an input-to-state stability of the control of Stefan problem with respect to an unknown heat loss. 

In this paper, a full-state feedback control law for the stabilization of the two-phase Stefan problem at a reference setpoint is studied. First, we state some assumptions and lemmas to guarantee the validity of the physical model required for the existence and uniqueness of the solution. Next, we design the control law by means of energy-shaping to satisfy the conditions of physical model. Then, we introduce a change of variable to absorb the solid phase dynamics to the ODE state to obtain the similar structure as the one-phase Stefan problem, and apply the backstepping transformation as in \cite{Shumon19journal}. The associated target system is shown to satisfy an exponential stability estimate in the ${ L}_{2}$ norm through the Lyapunov analysis. 

This paper is organized as follows. In Section \ref{sec:problem}, the two-phase Stefan problem is presented with stating some important remarks. Section \ref{sec:control} introduces the control problem statement and the control design via the energy-shaping and proves some conditions for the model validity under the closed-loop system. The main theorem of the stability of the closed-loop system and its proof are presented in Section \ref{sec:stability}. The robustness of the proposed control under uncertainty of the physical parameters is shown in Section \ref{sec:robust}. Supportive numerical simulations are provided in Section \ref{sec:simulation}. The paper ends with some final remarks and future directions  in  Section \ref{sec:conclusion}.

\section{Problem Statement}\label{sec:problem}
\begin{figure}[t]
\centering
\includegraphics[width=2.5in]{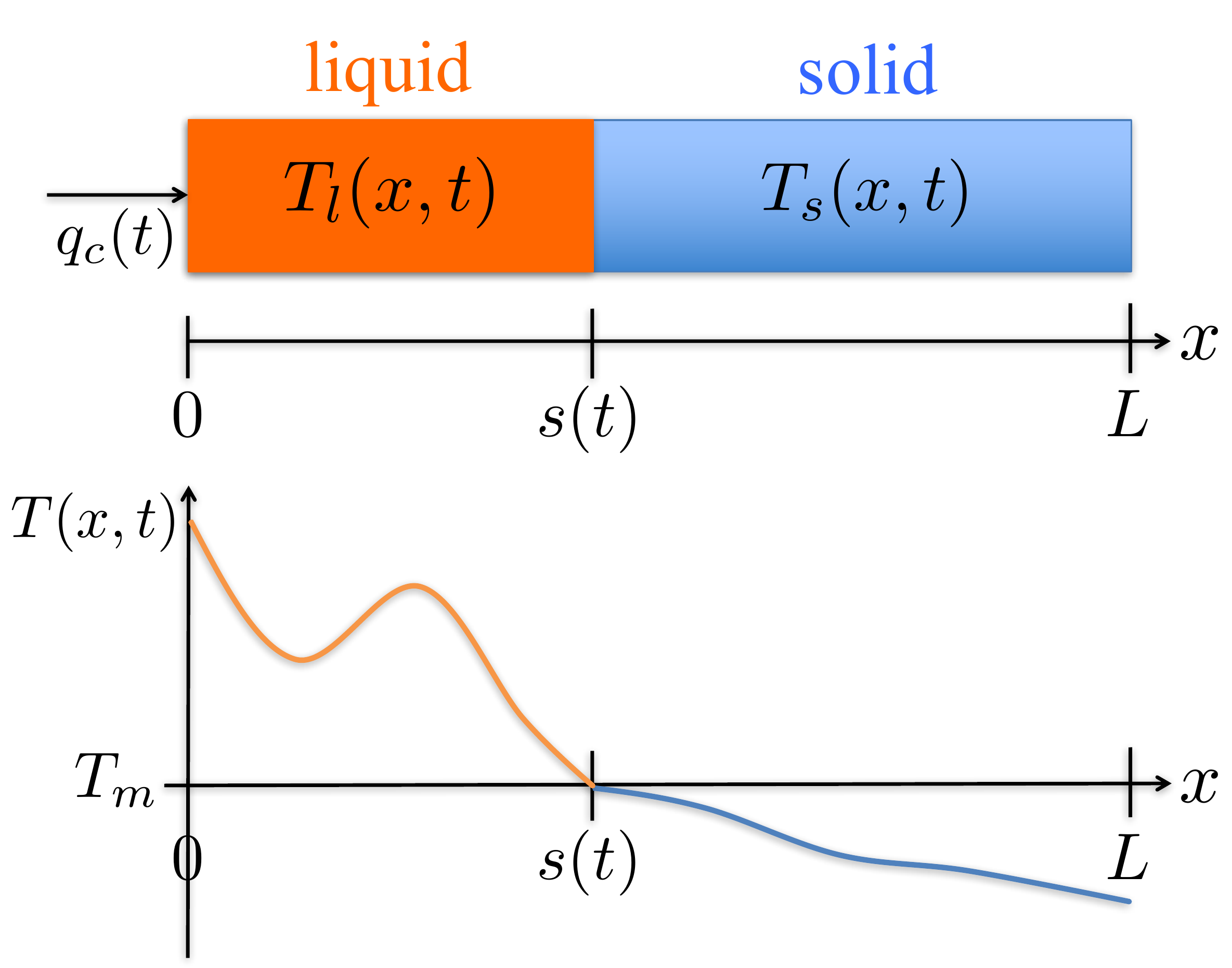}\\
\caption{Schematic of the two-phase Stefan problem.}
\label{fig:twophase}
\end{figure}

\subsection{Two-phase Stefan problem} 
The two-phase Stefan problem describes the thermodynamic model of the phase change phenomena such as melting or freezing (solidification) process in a pure material. The dynamics of the process depends strongly  on the evolution in time of the moving interface (here reduced to a point)  at which phase transition from liquid to solid (or equivalently, in the reverse direction) occurs. In this paper, we consider the one dimensional model with the material's length $L$, and the material's domain $x \in [0, L]$ is separated into two complementary time-varying sub-domains $x \in [0, s(t)]$ and $x \in [s(t), L]$ which are occupied by the liquid phase and the solid phase, respectively, as depicted in Fig. \ref{fig:twophase}. While the results in \cite{Shumon19journal} dealt with the one-phase Stefan problem by assuming the temperature in the solid phase to be steady-state, the two-phase Stefan problem describes both melting and solidification process by considering the temperature dynamics of both phases. Let $T_{{\rm l}}(x,t)$ and $T_{{\rm s}}(x,t)$ be the temperature profiles of liquid and solid, respectively, and $s(t)$ be the position of the interface between liquid and solid. Then, the energy conservation and heat transfer laws give the following PDE-ODE model of the temperature profile
\begin{align}\label{sys1}
 \fr{\pa T_{{\rm l}}}{\pa t}(x,t) =&\alpha_{{\rm l}}  \fr{\pa^2 T_{{\rm l}}}{\pa x^2}(x,t), \quad 0<x<s(t),\\
 \fr{\pa T_{{\rm s}}}{\pa t}(x,t) =&\alpha_{{\rm s}}  \fr{\pa^2 T_{{\rm s}}}{\pa x^2}(x,t), \quad s(t)<x<L,\label{sys2} \\
  \fr{\pa T_{{\rm l}}}{\pa x}(0,t) =& - \fr{q_{{\rm c}}(t)}{k_{{\rm l}}}, \quad \fr{\pa T_{{\rm s}}}{\pa x}(L,t) = 0, \label{BC:control}\\
T_{{\rm l}}(s(t),t) = &T_{{\rm m}},\quad T_{{\rm s}}(s(t),t) =T_{{\rm m}} \label{sys4} \\
\label{sys5} \gm \dot{s}(t) =& - k_{{\rm l}} \fr{\pa T_{{\rm l}}}{\pa x}(s(t),t)+k_{{\rm s}} \fr{\pa T_{{\rm s}}}{\pa x}(s(t),t),
\end{align}
with the initial data $T_{{\rm l},0}(x) := T_{{\rm l}}(x,0)$, $T_{{\rm s},0} (x) := T_{{\rm s}}(x,0)$, $s_0 := s(0)$, where $q_{{\rm c}}(t)>0$ is a boundary heat input. Here, $\alp_{i} = \fr{k_{i}}{\rho_i c_i}$, where $\rho_i$, $c_i$, $k_i$ for $i\in$\{l, s\} are the density, the heat capacity, the thermal conductivity, and the heat transfer coefficient, respectively and the subscripts ``l" and ``s" are associated to the liquid or solid phase, respectively. Also, $\gm = \rho_{{\rm l}} \Delta H^*$ where $\Delta H^*$ denotes the latent heat of fusion.

\subsection{Conditions to validate the physical model} \label{sec:valid} 

There are underlying assumptions to validate the model \eqref{sys1}-\eqref{sys5}. First, the liquid phase is not frozen to the solid phase from the boundary $x=0$. This condition is ensured if the liquid temperature $T_{{\rm l}}(x,t)$ is greater than the melting temperature $T_{{\rm m}}$. Second, in a similar manner, the solid phase is not melt to the liquid phase from the boundary $x=L$, which is ensured if the solid temperature $T_{{\rm s}}(x,t)$ is less than the melting temperature. Third, the material is not completely melt or frozen to single phase through the disappearance of the other phase. This condition is guaranteed if the interface position remains inside the material's domain. In addition, these conditions are also required for the well-posedness (existence and uniqueness) of the solution in this model. Taking into account of these model validity conditions, we emphasize the following remark. 
\begin{remark}\label{rem:valid}\emph{
To keep the physical state of each phase meaningful, the following conditions must be maintained: 
\begin{align}\label{valid1}
T_{{\rm l}}(x,t) \geq& T_{{\rm m}}, \quad \forall x\in(0,s(t)), \quad \forall t>0, \\
\label{valid2}T_{{\rm s}}(x,t) \leq& T_{{\rm m}}, \quad \forall x\in(s(t),L), \quad \forall t>0, \\
\label{valid3} 0< &s(t)<L, \quad \forall t>0. 
\end{align}}
\end{remark}

For model validity, we state the following assumption and lemma. 
\begin{assum}\label{initial} 
$0<s_0<L$, $T_{{\rm l},0}(x)$ and $T_{{\rm s},0}(x)$ are piecewise continuous functions, and there exist Lipschitz constants $ H_{{\rm l}}>0$ and $ H_{{\rm s}}>0$ such that 
\begin{align}
T_{{\rm m}}\leq T_{{\rm l},0}(x) \leq T_{{\rm m}} + H_{{\rm l}}(s_0 - x), \quad \forall x\in[0,s_0], \\
T_{{\rm m}}\geq T_{{\rm s},0}(x) \geq T_{{\rm m}} + H_{{\rm s}}(s_0 - x), \quad \forall x\in[s_0,L].
\end{align}
 \end{assum}
 \begin{lemma}\label{valid} 
Under Assumption \ref{initial}, and provided that $q_{{\rm c}}(t)$ is a piecewise continuous function that satisfies 
\begin{align} \label{qmqf} 
q_{{\rm c}}(t)\geq 0,\quad  \forall t \in [0, t^*), 
\end{align} 
there exists a finite time $\overline t:= \sup_{t \in (0, t^*)}\{t | s(t) \in (0,L)\}>0$ such that a classical solution to \eqref{sys1}--\eqref{sys5} exists, is unique, and satisfies the model validity condition \eqref{valid1}--\eqref{valid3} for all $t \in (0,\overline t)$. Moreover, if $t^* = \infty$ and it holds 
\begin{align} \label{validcondition} 
0< \gm s_{\infty} + \int_0^{t} q_{{\rm c}}(s) ds < \gm L ,
\end{align}
for all $t \geq 0$, where 
\begin{align}  \label{sinf}
s_{\infty}:= s_0 + \frac{k_{{\rm l}}}{\alpha_{{\rm l}} \gm} \int_{0}^{s_0} (T_{{\rm l},0}(x) - T_{{\rm m}}) {\rm d}x  + \frac{k_{{\rm s}}}{\alpha_{{\rm s}}\gm} \int_{s_0}^{L} (T_{{\rm s},0}(x) - T_{{\rm m}}) {\rm d}x	, 
\end{align}
then $\overline t = \infty$, namely, the well-posedness and the model validity conditions are satisfied for all $t \geq 0$.  
 \end{lemma}
 
Lemma \ref{valid} is proven in \cite{Cannon71flux} (Theorem 1 in p.4 and Theorem 4 in p.8) by employing the maximum principle. The variable $s_{\infty}$ defined in \eqref{sinf} is the final interface position $s_{\infty} = \lim_{t \to \infty} s(t)$ under the zero input $q_{{\rm c}}(t) \equiv 0$ for all $t \geq 0$.  For \eqref{validcondition} to hold for all $t \geq 0$, we at least require it to hold at $t = 0$, which leads to the following assumption. 
\begin{assum}\label{ass:E0}
The initial conditions that appear in $s_{\infty}$ in \eqref{sinf} satisfy 
\begin{align} 
0< s_{\infty} < L . 
\end{align}
\end{assum}


\section{State Feedback Control for Two-Phase Stefan Problem} \label{sec:control}


\subsection{Control problem statement and energy-shaping design} 
In this paper, we design the boundary heat input $q_{{\rm c}}(t)$ for the asymptotical stabilization of the interface position $s(t)$ at a desired reference setpoint $s_{{\mathrm r}}$. The steady-state solution for the temperature profiles at the desired setpoint of the system \eqref{sys1}--\eqref{sys5} is given by uniform melting temperature $T_{{\rm m}}$ for both liquid and solid phases. Thus, the control objective is to achieve the following convergences: 
\begin{align}
\lim_{t\to \infty} s(t) &= s_{{\mathrm r}},\label{c1}\\
\lim_{t\to \infty} T_{{\rm l}}(x,t) &= T_{{\mathrm m}}, \quad \lim_{t\to \infty} T_{{\rm s}}(x,t) = T_{{\mathrm m}}.\label{c2}
\end{align}
We approach to this problem by means of \emph{energy shaping control}, that is originally developed for underactuated mechanical systems such as robot manipulators \cite{Fantoni2000}. The thermal internal energy of the total system in \eqref{sys1}--\eqref{sys5} is given by 
\begin{align}\label{Et}
E(t) =& \frac{k_{{\rm l}}}{\alpha_{{\rm l}}} \int_{0}^{s(t)} (T_{{\rm l}}(x,t) - T_{{\rm m}}) {\rm d}x \notag\\
&+ \frac{k_{{\rm s}}}{\alpha_{{\rm s}}} \int_{s(t)}^{L} (T_{{\rm s}}(x,t) - T_{{\rm m}}) {\rm d}x + \gm s(t) , 
\end{align}
which includes the specific heat of both liquid and solid phases and the latent heat. Taking the time derivative of \eqref{Et} along the solution of \eqref{sys1}--\eqref{sys5}, one can obtain the energy conservation law formulated as  
\begin{align} \label{energy}
\frac{d}{dt} E(t) = q_{{\rm c}}(t) . 
\end{align} 
To achieve the control objective given by the conditions \eqref{c1}--\eqref{c2}, the internal energy \eqref{Et} must converges to the following setpoint energy  
\begin{align}\label{internal}
\lim_{t\to \infty} E(t) = \gm s_{{\rm r}}.
\end{align}
Taking the time integration of \eqref{energy} from $t=0$ to $\infty$, and imposing the input constraint \eqref{qmqf} required for the model validity as stated in Lemma \ref{valid}, in order to achieve \eqref{internal} we deduce that the following restriction on the setpoint \emph{neccesary}: 
\begin{assum}\label{assum2}
The setpoint $s_{\rm r}$ is chosen to satisfy
\begin{align} \label{setpoint}
s_{\infty} < s_{{\rm r}} < L,  
\end{align} 
where $s_{\infty}$ is defined in \eqref{sinf}. 
\end{assum}

With Assumption \ref{assum2}, due to the energy conservation \eqref{energy}, the following control law
\begin{align}
q_{{\rm c}}(t)&= -c ( E(t) - E_{{\rm r}}), \\
&= - c\left(\frac{k_{{\rm l}}}{\alpha_{{\rm l}}} \int_{0}^{s(t)} (T_{{\rm l}}(x,t) - T_{{\rm m}}) {\rm d}x \right. \notag\\
& \left. + \frac{k_{{\rm s}}}{\alpha_{{\rm s}}} \int_{s(t)}^{L} (T_{{\rm s}}(x,t) - T_{{\rm m}}) {\rm d}x +  \gamma (s(t) - s_{{\rm r}}) \right), \label{Fullcontrol}
\end{align}
drives the internal energy $E(t)$ to the reference energy $E_{{\rm r}}$. We study some properties of the closed-loop system under the control law \eqref{Fullcontrol}. 

\subsection{Some conditions required for model validity}

As addressed in Remark \ref{rem:valid}, the two-phase Stefan problem given by \eqref{sys1}--\eqref{sys5} must satisfy the conditions \eqref{valid1}--\eqref{valid3} under the state-feedback control law \eqref{Fullcontrol}. We state the following lemma. 
\begin{lemma} \label{lem:closed-valid} 
Under Assumptions \ref{initial}--\ref{assum2}, the closed-loop system of the plant \eqref{sys1}--\eqref{sys5} with control law \eqref{Fullcontrol} has a unique classical solution that satisfies the conditions \eqref{valid1}--\eqref{valid3} for model validity. Furthermore, it holds that 
\begin{align} \label{flux-valid} 
	\fr{\pa T_{{\rm l}}}{\pa x}(s(t),t) \leq 0, \quad \fr{\pa T_{{\rm s}}}{\pa x}(s(t),t) \leq 0, \quad \forall t \geq 0. 
\end{align}

\end{lemma}

\begin{proof}
By the energy-conservation \eqref{energy} and control law given by \eqref{Fullcontrol}, one can derive that the closed-loop system has the explicit solution for the state feedback control $q_{{\rm c}}(t)$ given by
\begin{align}  \label{qcsol}
q_{{\rm c}}(t) = c \gm (s_{{\rm r}} -s_{\infty} ) e^{-ct}. 
\end{align}
By Assumption \ref{assum2}, the solution \eqref{qcsol} yields the positivity of the heat input, i.e.,  \eqref{qmqf}. Moreover, substituting \eqref{qcsol} into the middle equation in \eqref{validcondition}, one can further show that the closed-loop system satisfies the inequalities \eqref{validcondition}, and thereby applying Lemma \ref{valid} leads to the well-posedness and the model validity conditions \eqref{valid1}-\eqref{valid3} to hold for all $t \geq 0$. Applying Hopf's lemma to the conditions \eqref{valid1}-\eqref{valid3} together with the boundary condition \eqref{sys4} leads to the conditions \eqref{flux-valid}. 
\end{proof}

\section{Stability Analysis}\label{sec:stability}

While the energy shaping method is utilized for the control design to stabilize the system's energy with ensuring the model validity conditions, the stability of the closed-loop system is proven by employing the backstepping method and Lyapunov analysis. We state our main theorem as follows. 

\begin{theorem}\label{Theo-1}
Under Assumptions \ref{initial}--\ref{assum2}, the closed-loop system consisting of the plant \eqref{sys1}--\eqref{sys5} and the control law \eqref{Fullcontrol} where $c>0$ is an arbitrary controller gain, maintains the conditions \eqref{valid1}--\eqref{valid3}, and there exists a positive constant $M>0$ such that the following exponential stability estimate holds: 
\begin{align}\label{H1norm}
& \Psi (t) \leq M  \Psi(0) e^{ - dt } 
\end{align}
for all $t\geq 0$, where $d = \frac{1}{2} \min\left\{ \fr{\alpha_{{\rm l}}}{2 L^2},  \fr{\alpha_{{\rm s}}}{ L^2}, c\right\} $, in the $L_2$-norm 
\begin{align} \label{Psidef}  
\Psi(t) =&  \int_{0}^{s(t)} (T_{{\rm l}}(x,t) - T_{{\rm m}})^2 {\rm d}x + \int_{s(t)}^{L} (T_{{\rm s}}(x,t) - T_{{\rm m}})^2 {\rm d}x \notag\\
&+ (s(t) - s_{{\rm r}})^2 . 
\end{align} 

\end{theorem}

	Due to the requirement of Assumption \ref{ass:E0}, the designed control ensures the closed-loop stability only when the initial condition does not cause a disappearance of one phase through complete melting or freezing under zero heat input.
%

For the setpoint position $s_{{\rm r}}$ violating Assumption \ref{assum2}, the control problem should be replaced from heating the liquid phase to cooling the solid phase. Mathematically, the boundary conditions \eqref{BC:control} are replaced by 
\begin{align} 
	  \fr{\pa T_{{\rm l}}}{\pa x}(0,t) =& 0, \quad \fr{\pa T_{{\rm s}}}{\pa x}(L,t) = \fr{q_{{\rm c}}(t)}{k_{{\rm s}}}, \label{BC:control2}
\end{align}
where the condition \eqref{qmqf} for the control input in Lemma \ref{valid} is replaced by $q_{{\rm c}}(t) \leq 0$ serving as cooling the solid phase. Then, owing to the symmetry of the system's structure, it is straightforward to show that Theorem \ref{Theo-1} is equivalent to the following corollary. 
\begin{corollary}
	Under Assumptions \ref{initial}--\ref{ass:E0}, and assuming $s_{{\rm r}} \in (0, s_{\infty})$, the closed-loop system consisting of the plant \eqref{sys1}--\eqref{sys2}, \eqref{sys4}--\eqref{sys5}, with boundary conditions \eqref{BC:control2}, and the control law \eqref{Fullcontrol}, maintains the conditions \eqref{valid1}--\eqref{valid3}, and there exists a positive constant $M>0$ such that the stability estimate \eqref{H1norm} holds in the $L_2$-norm \eqref{Psidef}.  
\end{corollary}

The proof of Theorem \ref{Theo-1} is established through several steps in the remainder of this section. 
 
\subsection{Error variables relative to melting temperature} 
We first introduce some change of variables so that the stabilization at the zero profiles should be achieved. Let $u(x,t)$, $v(x,t)$ be reference error temperature profiles of the liquid and the solid phase, respectively, defined as 
\begin{align}
u(x,t) = T_{\rm l} (x,t)-T_{{\rm m}}, \quad v(x,t) = T_{\rm s}(x,t)-T_{{\rm m}} . 
\end{align}
Then the system \eqref{sys1}--\eqref{sys5} is rewritten as 
\ba \label{u-sys1}
 u_{t}(x,t) =& \alpha_{{\rm l}} u_{xx}(x,t), \quad 0<x<s(t)\\
\label{u-sys2} u_x(0,t) =& - q_{{\rm c}}(t)/k_{{\rm l}}, \quad  u(s(t),t) = 0,\\
\label{vPDE} v_{t}(x,t) =& \alpha_{{\rm s}} v_{xx}(x,t), \quad s(t)<x<L\\
\label{vst}v_x(L,t) =& 0, \quad v(s(t),t) =  0, \\
 \dot{s}(t) =& - \beta_{\rm l} u_x(s(t),t)+ \beta_{\rm s} v_x(s(t),t), \label{u-sys5}
\end{align}
where $\beta_{i} = \frac{k_{i}}{\gm}$ for $i \in $\{l,s\}. The system \eqref{u-sys1}--\eqref{u-sys5} shows the two PDEs coupling with the ODE describing the moving boundary. The stabilization of states $(u,v,s)$ at $(0,0, s_{{\rm r}})$ is aimed by designing the control law, however, the multiple PDEs are difficult to deal with as themselves in general.  

\subsection{Change of variable to absorb the solid phase into the interface} 
To reduce the complexity of the system's structure in \eqref{u-sys1}--\eqref{u-sys5}, we introduce another change of variable. Let $X(t)$ be a state variable defined by 
\begin{align}\label{Xdef}
X(t) = s(t) - s_{{\rm r}} + \fr{\beta_{\rm s}}{\alpha_{{\rm s}}} \int_{s(t)}^{L} v(x,t) {\rm d}x . 
\end{align}
Taking the time derivative of \eqref{Xdef} and with the help of \eqref{vPDE}--\eqref{u-sys5}, we get $\dot{X}(t) =-  \beta_{\rm l} u_x(s(t),t)$ which eliminates $v$-dependency in ODE dynamics \eqref{u-sys5}. Thus, $(u,v,s)$-system in \eqref{u-sys1}--\eqref{u-sys5} can be reduced to $(u,X)$-system as  
\begin{align}\label{uX-sys1}
 u_{t}(x,t) =& \alpha_{{\rm l}} u_{xx}(x,t), \quad 0<x<s(t)\\
\label{uX-BC}u_x(0,t) = &- q_{{\rm c}}(t)/k_{{\rm l}}, \quad  u(s(t),t) = 0,\\
\dot{X}(t) =& - \beta_{\rm l} u_x(s(t),t). \label{uX-sys3}
\end{align}
Therefore, the control problem is now recast as designing the boundary control $q_{{\rm c}}(t)$ in \eqref{uX-BC} to stabilize the $(u,X)$-system in \eqref{uX-sys1}--\eqref{uX-sys3} at the zero states $(0,0)$, which is equivalent to the problem of stabilization of the one-phase Stefan problem studied in \cite{Shumon16}. The main difference with \cite{Shumon16} is that the monotonicity of the velocity of the moving interface, i.e. $\dot{s}(t)>0$, is not guaranteed in the two-phase Stefan problem due to the reversible melting and freezing process. This property was the key for stability proof in \cite{Shumon16}, and hence the same method cannot be directly applied to the two-phase problem considered in this paper. To resolve the issue, we modify the gain kernel function of the backstepping method for analyzing $L_2$ stability of the associated target system as in the next sections.  

\begin{remark}\label{rem_isidori} \emph{
Considering the output \eqref{Xdef} of the state-space system \eqref{u-sys1}--\eqref{u-sys5}, the equations \eqref{uX-sys1}--\eqref{uX-sys3}, along with \eqref{vPDE}, \eqref{vst}, can be considered as the system's input-output ``normal form" in the sense of Byrnes and Isidori~\cite{BI88,Isidori95} (see Chapter 4 in \cite{Isidori95}). As indicated above, a control design will be conducted for the input-output dynamics \eqref{uX-sys1}--\eqref{uX-sys3}, i.e., the dynamics of the liquid phase with a modified interface output map $X=h(s,v)$ using the control $q_{\rm c}$. The stability of the inverse dynamics \eqref{vPDE}, \eqref{vst}, which happen to be the dynamics of the solid phase, is studied in Section \ref{sec:stability-solid}. }
\end{remark}

\subsection{Backstepping transformation} 
Consider the following backstepping transformation and the gain kernel function $\phi$ given by
\begin{align}\label{bkst}
w(x,t)=&u(x,t)-\frac{\beta_{\rm l}}{\alpha_{{\rm l}}} \int_{x}^{s(t)} \phi (x-y)u(y,t) dy \notag\\
&-\phi(x-s(t)) X(t), \\
\phi(x) =& \frac{1}{\beta_{\rm l}} (cx- \ep) , 
\end{align} 
where $\ep >0$ is a parameter to be determined in the stability analysis. Taking the spatial and the time derivative of \eqref{bkst} along the solution of \eqref{uX-sys1}--\eqref{uX-sys3}, the associated target system is derived as
\begin{align}\label{eq:tarPDE1}
w_t(x,t)=&\alpha_{{\rm l}} w_{xx}(x,t)+ \frac{c}{\beta_{\rm l}}\dot{s}(t) X(t), \\
\label{eq:tarBC11} w(s(t),t) =& \frac{\ep}{\beta_{\rm l}} X(t), \\
\label{eq:tarODE1}\dot{X}(t)=&-cX(t)- \beta_{\rm l}w_x(s(t),t). 
\end{align}
Taking the derivative of \eqref{bkst} in $x$, we obtain 
\begin{align}\label{wx}
w_{x}(x,t)=&u_{x}(x,t)- \fr{\ep }{\alpha_l} u(x,t) \notag\\
&-\frac{c}{\alpha_{{\rm l}}} \int_{x}^{s(t)} u(y,t) dy- \frac{c}{\beta_{\rm l}} X(t). 
\end{align}
For a standard backstepping procedure, the boundary condition at $x=0$ of the target system leads to the control design. If we chose $w_{x}(0,t) = 0$, we obtain a stable target system in the case of fixed domain. However, the control design derived from $   w_{x}(0,t) = 0$ does not ensure the positivity and the required conditions addressed in Lemma \ref{valid1}. As proposed in Section \ref{sec:control}, we design the control law \eqref{Fullcontrol} by means of energy shaping to guarantee the required conditions, which is rewritten as  
\begin{align} 
q_{{\rm c}}(t) =& -c \left( \fr{k_{{\rm l}}}{\alpha_{{\rm l}}}  \int_{0}^{s(t)} u(y,t) dy +  \frac{k_{{\rm l}}}{\beta_{\rm l}}X(t) \right), \label{qcfbk2}
\end{align}
with respect to $X(t)$, and obtain the boundary condition of the target system. Setting $x=0$ in \eqref{wx} and applying \eqref{qcfbk2}, the boundary condition at $x=0$ is obtained by 
\begin{align}\label{eq:tarBC21}
w_{x}(0,t) = - \fr{\ep }{\alpha_l} u(0,t), 
\end{align}
of which the right hand side should be rewritten with respect to $(w, X)$ after we derive the inverse transformation. Here, we note the following lemma. 
\begin{lemma} \label{lem:Xt}
	It holds $X(t) \leq 0$ for all $t \geq 0$. 
\end{lemma}
Lemma \ref{lem:Xt} is shown straightforwardly by applying \eqref{valid1} and \eqref{qmqf} to \eqref{qcfbk2}.

\subsection{Inverse transformation} 
Suppose that the inverse transformation is described by the following formulation
\begin{align}\label{inv}
u(x,t)=&w(x,t)-\frac{\beta_{\rm l}}{\alpha_{{\rm l}}} \int_{x}^{s(t)} \psi (x-y)w(y,t) dy \notag\\
&-\psi(x-s(t)) X(t), 
\end{align}
where $\psi$ is a gain kernel function to be determined. Taking derivatives of \eqref{inv} in $x$ and $t$ along the solution of the target system \eqref{eq:tarPDE1}--\eqref{eq:tarODE1}, in order to satisfy \eqref{uX-sys1}--\eqref{uX-sys3}, one can show that the gain kernel function must satisfy the following 
\begin{align}\label{psiODE}
 \alpha_{{\rm l}} \psi''(x) -   \ep  \psi' (x)   + c \psi(x)   = 0, \\
\label{psiBC} \psi(0) = \frac{\ep}{\beta_{\rm l}} , \quad \psi'(0) = \frac{\ep ^2}{\alpha_{{\rm l}} \beta_{\rm l}} - \frac{c}{\beta_{\rm l}}. 
\end{align}
Supposing $\ep < \sqrt{2 \alpha_{{\rm l}} c}$, the solution to the differential equation \eqref{psiODE}--\eqref{psiBC} is derived as  
\begin{align}\label{psisol}
\psi(x) = e^{  r x } \left( p_1 \sin\left( \omega x \right) + p_2 \cos\left( \omega x \right) \right) , 
\end{align}
where $
r = \fr{ \ep }{2 \alpha_{{\rm l}}}$, $\omega = \sqrt{\fr{4 \alpha_{{\rm l}} c - \ep ^2 }{4 \alpha_{{\rm l}}^2 } }$ , $p_1  = - \fr{1}{2 \alpha_{{\rm l}} \beta_{\rm l}\omega } \left( 2 \alpha_{{\rm l}} c - \ep^2 \right) $, and $p_2 = \frac{\ep}{\beta_{\rm l}} $. 
The solution \eqref{psisol} satisfies 
\begin{align} \label{psiineq} 
\psi(x) ^2  \leq  \frac{4 \alpha_{{\rm l}} c}{\beta_{\rm l}^2} e^{ 2 r x },
\end{align} 
which is used in the stability analysis. Finally, by applying the inverse transfotmation, the boundary condition \eqref{eq:tarBC21} is described by only with respect to the target state $(w,X)$ as  
\begin{align}\label{tarBC22}
w_{x}(0,t) = & - \frac{\ep }{\alpha_{{\rm l}}} \left[  w(0,t)-\frac{\beta_{\rm l}}{\alpha_{{\rm l}}} \int_{0}^{s(t)} \psi (-y)w(y,t) dy \right. \notag\\
&\left. -\psi(-s(t)) X(t) \right] . 
\end{align}
Therefore, the target $(w, X)$-system is written as \eqref{eq:tarPDE1}--\eqref{eq:tarODE1} and \eqref{tarBC22} as a closed form. Note that this target $(w, X)$-system is not a standard choice due to its complicated structure through the coupling between each state. Nevertheless, the target system is proven to satisfy the exponential stability estimate in $L_2$ norm in the next section with the help of the properties in Lemma \ref{lem:closed-valid}.

\subsection{Lyapunov method} 
Due to the invertibility of the transformations \eqref{bkst} and \eqref{inv}, the $(u,X)$-system in \eqref{uX-sys1}--\eqref{uX-sys3} with the control law \eqref{qcfbk2} has the equivalent stability property with the target $(w, X)$-system in \eqref{eq:tarPDE1}--\eqref{eq:tarODE1}, \eqref{tarBC22}. Owing to the stabilizing term $- cX(t)$ in \eqref{eq:tarODE1}, the exponential stability of $(w, X)$-system is shown via Lyapunov analysis using the conditions verified in the lemmas, and the stability of the overall system is proven. 

\subsection{Stability analysis for the liquid with modified interface} 

Consider the following Lyapunov functional
\begin{align}\label{Vdef}
V (t)= &\fr{1}{2\alpha_{{\rm l}} } || w ||^2  + \fr{\ep}{2 \beta_{\rm l}^2} X(t)^2,  
\end{align}
where the $L_2$ norm is denoted as $|| w || : = \sqrt{ \int_0^{s(t)} w(x,t)^2 {\rm d}x} $. Taking the time derivative of \eqref{Vdef} along the solution of \eqref{eq:tarPDE1}--\eqref{eq:tarODE1}, \eqref{tarBC22}, we get 
\begin{align} 
\dot{V}(t)=&  - || w_{x}||^2  - \ep \frac{c}{\beta_{\rm l}^2}X(t)^2  +\frac{\ep}{\alpha_{{\rm l}}} w(0,t)^2  \notag\\
& -\frac{\ep}{\alpha_{{\rm l}}} w(0,t)\left[ \frac{\beta_{\rm l}}{\alpha_{{\rm l}}} \int_{0}^{s(t)} \psi (-y)w(y,t) dy\right.  \notag\\  
&\left. + \psi(-s(t)) X(t) \right] \notag\\
\label{Vdot1}&+\fr{ \dot{s}(t)}{2 \alpha_{{\rm l}}} \left( \left(\frac{\ep}{\beta_{\rm l}} X(t)\right)^2 +  2 \frac{c}{\beta_{\rm l}} \int_0^{s(t)} w(x,t) {\rm d}x X(t)  \right) .
\end{align} 
 Applying Young's, Cauchy Schwarz, Poincare, and Agmon's inequalities with the help of $0<s(t)<L$, and the inequality of $\psi$ in \eqref{psiineq}, we get the following  
 \begin{align}\label{Vdot4}
 \dot{V} (t)\leq &  - \left( 1 - \frac{2 \ep L}{\alpha_{{\rm l}}} \left( 3 + \fr{32 c L^2}{\alpha_{{\rm l}}} \right) \right) || w_{x}||^2 \notag\\
 & - \frac{\ep}{\beta_{\rm l}^2} \left(\fr{c}{2} - \frac{ \ep^2}{\alpha_{{\rm l}}} \left( 3 + \fr{32 c L^2}{\alpha_{{\rm l}}} \right) \right)X(t)^2  \notag\\
&+\fr{ \dot{s}(t)}{2 \alpha_{{\rm l}}} \left( \left(\frac{\ep}{\beta_{\rm l}} X(t)\right)^2 +  2 \frac{c}{\beta_{\rm l}} \int_0^{s(t)} w(x,t) {\rm d}x X(t)  \right)  .
\end{align}
Let us choose $\ep$ to satisfy $ \ep < \ep_{1} := \frac{\alpha_{{\rm l}}}{4 L  \left( 3 + \fr{32 c L^2}{\alpha_{{\rm l}}} \right) }$. Then, applying Poincare's inequality to the first term of \eqref{Vdot4} again, we get 
\begin{align}\label{Vdot6}
 \dot{V} (t)\leq &  - \fr{1}{8L^2} || w||^2  - \frac{\ep}{\beta_{\rm l}^2} \left(\fr{c}{4} + g(\ep)\right)X(t)^2  \notag\\
&+\fr{ \dot{s}(t)}{2 \alpha_{{\rm l}}} \left( \left(\frac{\ep}{\beta_{\rm l}} X(t)\right)^2 +  2 \frac{c}{\beta_{\rm l}} \int_0^{s(t)} w(x,t) {\rm d}x X(t)  \right)  ,
\end{align}
where $g(\ep) := \fr{c}{4} - \fr{\ep}{4 L} - \frac{\ep^2 }{\alpha_{{\rm l}}} \left( 3 +  \fr{32 c L^2}{\alpha_{{\rm l}}} \right)$. Since $g(0) = \fr{c}{4}>0$ and $g'(\ep) = - \fr{1}{4 L} - \frac{2\ep }{\alpha_{{\rm l}}} \left( 3 + \fr{ 32 c L^2}{\alpha_{{\rm l}}} \right) <0$ for all $ \ep>0$, there exists $\ep^*$ such that $g(\ep)>0$ for all $\ep \in (0, \ep^*)$ and $g(\ep^*)=0$. By choosing $\ep = \min \left\{\ep_1,  \ep^* \right \} $, we obtain 
\begin{align}\label{Vdot7}
 \dot{V} (t) \leq &  - \fr{1}{8L^2} || w||^2  - \fr{c\ep }{4 \beta_{\rm l}^2}X(t)^2 \notag\\
 &+\fr{ \left| \dot{s}(t)\right|}{2 \alpha_{{\rm l}} } \left( \left(\frac{\ep}{\beta_{\rm l}} X(t)\right)^2 +  2 \frac{c}{\beta_{\rm l}} \left | \int_0^{s(t)} w(x,t) {\rm d}x X(t) \right| \right) . 
\end{align}
Since $u_{x}(s(t),t)<0$ and $v_x(s(t),t)<0$ by \eqref{flux-valid}, we have 
\begin{align} 
| \dot{s}(t)| \leq - \beta_{\rm l} u_{x}(s(t),t) - \beta_{\rm s} v_x(s(t),t) . 
\end{align} 
Introduce 
\begin{align} \label{ztdef}
z(t) := X(t)+ \fr{\beta_{\rm s}}{\alpha_{{\rm s}}} \int_{s(t)}^{L} v(x,t) {\rm d}x <0,
\end{align}  
where the negativity follows from Lemmas \ref{lem:closed-valid} and \ref{lem:Xt}. Taking the time derivative of \eqref{ztdef} yields 
\begin{align} 
\dot{z}(t) = - \beta_{\rm l} u_{x}(s(t),t) - \beta_{\rm s} v_x(s(t),t) >0,
\end{align}  
where the positivity follows from $u_{x}(s(t),t)<0$ and $v_{x}(s(t),t)<0$ in Lemma \ref{lem:closed-valid}. Applying this inequality and Young's and Cauchy Schwarz inequalities to the last term of \eqref{Vdot7}, we arrive at
\begin{align}\label{Vdot9}
 \dot{V} (t)\leq & - bV (t)+ a \dot{z}(t) V(t), 
\end{align}
where $b = \min \left \{ \fr{\alpha_{{\rm l}}}{4 L^2}, \fr{c}{2} \right\} $, $a = \fr{ 1}{2 \alpha_{{\rm l}} }  \max \left\{  \fr{2 \alpha_{{\rm l}} c^2L}{\ep^2} , 4 \ep \right\} $. Consider the functional $
W(t) = V (t)e^{ - a z(t)} $. 
Taking the time derivative and applying \eqref{Vdot9}, one can deduce 
\begin{align} 
\dot{W} (t)= \left( \dot{V} (t)- a \dot{z}(t) V(t) \right) e^{ - a z(t)} \leq - b W(t) . 
\end{align} 
Hence, $W(t) \leq W_0 e^{-bt} $ is satisfied, which leads to 
\begin{align}\label{Vtexp}
V(t) \leq e^{a (z(t) - z(0)) } V_0 e^{-bt} \leq \delta  V_0 e^{-bt} , 
\end{align}
where $\delta $ is defined as a constant which bounds $\delta > e^{ -a z(0)}$, of which the existence is ensured by Assumptions \ref{initial}-\ref{assum2} and properties proven in the lemmas. Let $V_1(t)$ be the functional defined by 
\begin{align} 
V_1(t) = || u ||^2 = \int_0^{s(t)} u(x,t)^2 {\rm d}x.
\end{align}  
Due to the invertibility of the transformations \eqref{bkst} and \eqref{inv}, there exist positive constants $\underline{M}>0$, $\bar M>0$ such that the following norm equivalence between $(u, X)$-system and $(w, X)$-system holds: 
\begin{align} 
\underline{M} \left(V_1(t) + X(t)^2 \right) \leq V(t) \leq \bar{M} \left(V_1(t) + X(t)^2 \right). 
\end{align} 
Hence, by \eqref{Vtexp}, the following exponential stability estimate of the $(u, X)$-system is shown: 
\begin{align}\label{V1exp} 
V_1(t) + X(t)^2 \leq &  \fr{\bar{M}}{\underline{M}} \delta  \left( V_1(0) + X(0) \right) e^{-bt}. 
\end{align}
\subsection{Stability analysis for the solid phase} \label{sec:stability-solid} 
Let $V_2(t)$ be the functional defined by 
\begin{align} 
V_2 (t) = || v ||^2 = \int_{s(t)}^{L} v(x,t)^2 {\rm d}x.
\end{align}  
Taking the time derivative along the solution of \eqref{vPDE}--\eqref{vst} (note $v(s(t),t) =  0$), and applying Poincare's inequality with the help of $0 < s(t) < L$, we obtain 
\begin{align}
\dot{V}_2 (t) = &- \dot s(t) v(s(t),t)^2 - 2 \alpha_{{\rm s}} \int_{s(t)}^{L} v_{x}(x,t)^2 {\rm d}x \notag\\
\leq & - \fr{\alpha_{{\rm s}}  }{2 (L-s(t))^2} V_2(t) < - \fr{\alpha_{{\rm s}}}{2 L^2} V_2(t) . \label{V2dot}
\end{align}
By comparison principle, the differential inequality \eqref{V2dot} yields
\begin{align}\label{V2exp}
V_2(t) \leq V_2(0) e^{ - \fr{\alpha_{{\rm s}}}{2 L^2} t}. 
\end{align}

As announced in Remark \ref{rem_isidori}, by \eqref{V2exp} we have proven that the inverse dynamics given by \eqref{vPDE}, \eqref{vst}, with $s(t)$ expressed as the solution of \eqref{Xdef} in terms of $X(t)$ is exponentially stable, robustly in the input $X$ of these inverse dynamics. 

\subsection{Stability of overall liquid-interface-solid system} \label{sec:allstability} 

Applying Young's and Cauchy Schwartz inequalities to the square of \eqref{Xdef} with the help of $0<s(t)<L$ yields 
\begin{align} \label{Xineq}
 X(t)^2 \leq & 2 V_3(t) + \fr{2 L \beta_{\rm s}^2}{\alpha_{{\rm s}}^2} V_2(t) , 
\end{align}
where we defined $ V_3(t) = |s(t) - s_{{\rm r}}|^2$. On the other hand, the bound of $V_3(t)$ with respect to $X(t)^2$ and $V_2(t)$ are also obtained in the similar manner to \eqref{Xineq}, which yields  
\begin{align}\label{Ytineq}
 Y(t)   \leq  2 X(t)^2 + \fr{2 L \beta_{\rm s}^2}{\alpha_{{\rm s}}^2} V_2 (t). 
\end{align}
Finally, summing the norms of the liquid temperature, the interface position, and the solid temperature, respectively, and applying \eqref{V1exp}--\eqref{Ytineq},  we can see that there exists a positive constant $M$ such that 
\begin{align} 
V_1(t) + Y(t) + V_2(t) \leq  M  \left( V_1(0) + Y(0) + V_2(0) \right) e^{- \min\left\{ b,  \fr{\alpha_{{\rm s}}}{2 L^2}\right\} t},
\end{align}  
which completes the proof of Theorem \ref{Theo-1}. 

\section{Robustness to Uncertainties of Physical Parameters} \label{sec:robust}
The control design \eqref{Fullcontrol} requires the physical parameters of both the liquid and solid phases, however, in practice these parameters are uncertain. Guaranteeing the robustness of the stability of the closed-loop system with respect to such parametric uncertainties is significant. Suppose that the proposed control law is replaced by 
\begin{align} \label{robust_control} 
q_{{\rm c}}(t)=	-c \bigg( &\frac{k_{{\rm l}}}{\alpha_{{\rm l}}} (1 + \ep_{\rm l})\int_{0}^{s(t)} (T_{{\rm l}}(x,t) - T_{{\rm m}}) {\rm d}x \notag\\
& + \frac{k_{{\rm s}}}{\alpha_{{\rm s}}}(1 + \ep_{\rm s}) \int_{s(t)}^{L} (T_{{\rm s}}(x,t) - T_{{\rm m}}) {\rm d}x 
\notag \\ & +  \gamma (1 + \ep_{\rm f}) (s(t) - s_{{\rm r}}) \bigg), 
\end{align}
where $\ep_{\rm l}$, $\ep_{\rm s}$, and $\ep_{\rm f}$ are the uncertainties of physical parameters satisfying $\ep_{\rm l}>-1$, $\ep_{\rm s} \geq -1$, and $\ep_{\rm f} \geq -1$. We state the following theorem. 
\begin{theorem} \label{thm-robust}
	Under Assumptions \ref{initial}, \ref{ass:E0}, and assuming that the setpoint is chosen to satisfy $q_{{\rm c}}(0) > 0$ with \eqref{robust_control} and $s_{{\rm r}} < L$, consider the closed-loop system consisting of the plant \eqref{sys1}--\eqref{sys5} and the control law \eqref{robust_control}. Then, for any perturbations $(\ep_{\rm l}, \ep_{\rm s}, \ep_{\rm f})$ satisfying 
	\begin{align} \label{robust_cond} 
	\ep_{\rm l} \geq \ep_{\rm f} \geq \ep_{\rm s}, 
	\end{align}
	there exists $R>0$ such that if 
	\begin{align} 
		\left| \frac{ \ep_{\rm f}-\ep_{\rm l}}{1 + \ep_{\rm l}} \right| < R, 
	\end{align}
then the closed-loop system maintains model validity \eqref{valid1}-\eqref{valid3} and the exponential stability at the origin holds for the norm defined in \eqref{Psidef}.  
	\end{theorem}

Theorem \ref{thm-robust} implies that if we know lower and upper bounds of the physical parameters as $\underline{k}_{\rm l} \leq k_{{\rm l}} \leq \overline{k}_{\rm l}$, $\underline{\alp}_{\rm l} \leq \alp_{\rm l} \leq \overline{\alp}_{\rm l}$, and $\underline{\gm} \leq \gm \leq \overline{\gm}$, then the most conservative choice of the control law to satisfy the condition \eqref{robust_cond} is given by 
\begin{align} 
q_{{\rm c}}(t) = - c \left(	\frac{\overline{k}_{\rm l}}{\underline{\alpha}_{\rm l}} \int_{0}^{s(t)} (T_{{\rm l}}(x,t) - T_{{\rm m}}) {\rm d}x +  \underline{\gamma} (s(t) - s_{{\rm r}}) \right),  
\end{align}
which does not incorporate the solid phase temperature. This design requires less information than the exact feedback design \eqref{Fullcontrol}, however, the conditions $q_{{\rm c}}(0)> 0$ and $s_{{\rm r}} < L$, which lead to 
\begin{align} 
	s_0 + \frac{\overline{k}_{\rm l}}{\underline{\alpha}_{\rm l}\underline{\gamma}} \int_{0}^{s_0} (T_{{\rm l},0}(x) - T_{{\rm m}}) {\rm d}x  < s_{{\rm r}} < L, 
\end{align}
are more restrictive than Assumption \ref{assum2} for the unperturbed design \eqref{Fullcontrol}, which causes a tradeoff between the parameters' uncertainty and the restriction of the setpoint.   

The proof of Theorem \ref{thm-robust} is established by following similar steps to Section \ref{sec:stability}.  

\begin{proof} 
First, we derive an analogous result on the properties of the closed-loop system to Lemma \ref{lem:closed-valid} by employing contradiction approach twice. Assume that there exists a finite time $ t^*>0$ such that $q_{{\rm c}}(t) >0$ for all $t \in [0, t^*)$ and $q_{{\rm c}}(t^*) = 0$. Then, by Lemma \ref{valid}, 
for $t \in (0, \bar t)$ where $\bar t:= \sup_{t\in(0,t^*)}\{t|s(t) \in (0,L)\}$, the solution exists and unique with satisfying \eqref{valid1}--\eqref{valid3}. If $\bar t < t^*$, then it implies $s(\bar t) = 0$ or $s(\bar t) = L$ hold. However, under Assumption \ref{ass:E0} and $q_{{\rm c}}(t) >0$ for all $t \in [0, t^*)$, $s(t) > s_{\infty}>0$ holds for all $t \in (0, t^*)$, and hence $s(\bar t) \neq 0$. Moreover, applying $q_{{\rm c}}(t) > 0$ and \eqref{valid1} and \eqref{valid2} for all $t \in (0, \bar t)$ to the feedback design \eqref{robust_control}, one can see that $s(\bar t) \neq L$. Hence, $\bar t = t^*$. Taking the time derivative of the control law \eqref{robust_control}, we get the following differential equation: 
\begin{align} 
\dot{q}_{\rm c}(t)=& - c (1 + \ep_{\rm l}) q_{{\rm c}}(t)	 - (\ep_{\rm l}-\ep_{\rm f}) c k_{{\rm l}} \frac{\pa T_{{\rm l}}}{\pa x}(s(t),t) \notag\\
&+ (\ep_{\rm s} - \ep_{\rm f}) c k_{{\rm s}} \frac{\pa T_{{\rm s}}}{\pa x}(s(t),t), \label{robust_qcder}\\
\geq  & - c (1 + \ep_{\rm l}) q_{{\rm c}}(t), \quad \forall t \in (0, t^*) , \label{robust_qcineq}
\end{align}
where the inequality from \eqref{robust_qcder} to \eqref{robust_qcineq} follows from \eqref{robust_cond} and Hopf's lemma with the help of \eqref{valid1} and \eqref{valid2} for all $ t \in (0, t^*)$. Therefore, applying comparison principle to \eqref{robust_qcineq}, one can show that $q_{{\rm c}}(t) > q_{{\rm c}}(0) e^{-ct}$ for all $t \in (0, t^*)$, which leads to the contradiction with the imposed assumption $q_{{\rm c}}(t^*) = 0$. Thus, there does not exist such $t^*$, from which we conclude $q_{{\rm c}}(t) \geq 0$ for all $ t \geq 0$ and the well-posedness and the conditions \eqref{valid1}--\eqref{valid3} holds for all $t \geq 0$.

Next, we prove the stability of the perturbed closed-loop in the similar manner as the proof of Theorem \ref{Theo-1}. Let $\bar c = c (1 + \ep_{\rm l})$, and redefine the gain kernel function as $\phi = \frac{1}{\beta_{\rm l}} (\bar cx- \ep) $ associated with the backstepping transformation \eqref{bkst}. Then, the target systems is described as  
\begin{align}\label{robust-tarPDE1}
w_t(x,t)=&\alpha_{{\rm l}} w_{xx}(x,t)+ \frac{\bar c}{\beta_{\rm l}}\dot{s}(t) X(t), \\
\label{robust-tarBC11} w(s(t),t) =& \frac{\ep}{\beta_{\rm l}} X(t), \\
\label{robust-tarODE1}\dot{X}(t)=&- \bar cX(t)- \beta_{\rm l}w_x(s(t),t), 
\end{align}
and the boundary condition at $x = 0$ is given by 
\begin{align} \label{robust-tarBC22}
w_{x}(0,t) = & - \frac{\ep }{\alpha_{{\rm l}}}  u(0,t) +d(t) , 
\end{align} 
where $d(t)$ is the perturbation caused by the parametric uncertainties, given by 
\begin{align} 
d(t) = 	\frac{\ep_{\rm s}-\ep_{\rm f}}{1 + \ep_{\rm l}} \fr{ \bar c k_{\rm s}}{k_{\rm l} \alp_{\rm s}} \int_{s(t)}^L v(x,t) {\rm d}x + \frac{\ep_{\rm f}-\ep_{\rm l}}{1 + \ep_{\rm l}} \frac{\bar c }{k_{\rm l}} X(t). \label{dtdef}
\end{align}
We consider the Lyapunov function defined by \eqref{Vdef}. Using the same technique as the derivation of \eqref{Vdot4}, the time derivative of $V (t)= \fr{1}{2\alpha_{{\rm l}} } || w ||^2  + \fr{\ep}{2 \beta_{\rm l}^2} X(t)^2$ along the perturbed target system \eqref{robust-tarPDE1}--\eqref{robust-tarBC22} satisfies the following inequality 
\begin{align} 
\dot{V} \leq & - \left( 1 - \frac{2 \ep L}{\alpha_{{\rm l}}} \left( 3 + \fr{32 \bar c L^2}{\alpha_{{\rm l}}} \right) \right) || w_{x}||^2 \notag\\
 & - \frac{\ep}{\beta_{\rm l}^2} \left(\fr{\bar c}{2} - \frac{ \ep^2}{\alpha_{{\rm l}}} \left( 3 + \fr{32 \bar c L^2}{\alpha_{{\rm l}}} \right) \right)X(t)^2  - w(0,t) d(t), \notag\\
&+\fr{ \dot{s}(t)}{2 \alpha_{{\rm l}}} \left( \left(\frac{\ep}{\beta_{\rm l}} X(t)\right)^2 +  2 \frac{ \bar c}{\beta_{\rm l}} \int_0^{s(t)} w(x,t) {\rm d}x X(t)  \right) . \label{robust_Vdot4}
\end{align} 
Applying Young's and Agmon's inequalities, the perturbation is bounded by  
\begin{align} 
	- w(0,t) d(t) \leq & \frac{1}{8L} w(0,t)^2 + 2L d(t)^2, \notag\\
	\leq &\frac{1}{4 L} w(s(t),t)^2 + \frac{1}{2} ||w_{x}||^2 + 2 L d(t)^2 , \notag\\
	\leq & \frac{\ep^2}{4 L} X(t)^2 + \frac{1}{2} ||w_{x}||^2 + 2 L d(t)^2. \label{Y-ag}
\end{align}
Moreover, applying Young's and Cauchy Schwarz inequalities to the square of \eqref{dtdef}, we get 
\begin{align} 
	d(t)^2 = & 2	\left(\frac{\ep_{\rm s}-\ep_{\rm f}}{1 + \ep_{\rm l}} \fr{ \bar c k_{\rm s}}{k_{\rm l} \alp_{\rm s}} \right)^2 L \int_{s(t)}^L v(x,t)^2 {\rm d}x \notag\\
	& + 2 \left(\frac{\ep_{\rm f}-\ep_{\rm l}}{1 + \ep_{\rm l}} \frac{\bar c }{k_{\rm l}} \right)^2X(t)^2. \label{Y-cau}
\end{align}
Applying \eqref{Y-ag} and \eqref{Y-cau} to \eqref{robust_Vdot4}, we can see that there exists sufficiently small $\ep >0$ such that the following inequality holds 
\begin{align}
 \dot{V} \leq &  - \fr{1}{16L^2} || w||^2  - \bar c \left( \frac{\ep }{4 \beta_{\rm l}^2} - \frac{4 L \bar c }{ k_{\rm l}^2} \left| \frac{ \ep_{\rm f}-\ep_{\rm l}}{1 + \ep_{\rm l}} \right|^2  \right)X(t)^2 \notag\\
 & +4 L^2 \left(\frac{\ep_{\rm s}-\ep_{\rm f}}{1 + \ep_{\rm l}} \fr{ \bar c k_{\rm s}}{k_{\rm l} \alp_{\rm s}} \right)^2 \int_{s(t)}^L v(x,t)^2 {\rm d}x  \notag\\
&+\fr{ \dot{s}(t)}{2 \alpha_{{\rm l}}} \left( \left(\frac{\ep}{\beta_{\rm l}} X(t)\right)^2 +  2 \frac{c}{\beta_{\rm l}} \int_0^{s(t)} w(x,t) {\rm d}x X(t)  \right)  .  \label{robust-Vineq9} 
\end{align}
Therefore, if
\begin{align} 
	\left| \frac{ \ep_{\rm f}-\ep_{\rm l}}{1 + \ep_{\rm l}} \right|^2 < \frac{\ep k_{\rm l}^2}{32 \beta_{\rm l}^2 L \bar c} , 
\end{align}
then the differential inequality \eqref{robust-Vineq9} is led to 
\begin{align} 
\dot{V} \leq &  - \fr{1}{16L^2} || w||^2  -  \frac{\ep \bar c }{8 \beta_{\rm l}^2} X(t)^2 \notag\\
 & +4 L^2 \left(\frac{\ep_{\rm s}-\ep_{\rm f}}{1 + \ep_{\rm l}} \fr{ \bar c k_{\rm s}}{k_{\rm l} \alp_{\rm s}} \right)^2 \int_{s(t)}^L v(x,t)^2 {\rm d}x  \notag\\
&+\fr{ \dot{s}(t)}{2 \alpha_{{\rm l}}} \left( \left(\frac{\ep}{\beta_{\rm l}} X(t)\right)^2 +  2 \frac{c}{\beta_{\rm l}} \int_0^{s(t)} w(x,t) {\rm d}x X(t)  \right)  . \label{robust_Vdotfin} 
\end{align}
Since $v$-system is equivalent to the one in previous sections, the time derivative of $V_2 = ||v||^2$ satisfies the inequality \eqref{V2dot}, which is 
\begin{align}
\dot{V}_2 \leq - \fr{\alpha_{{\rm s}}}{2 L^2} V_2 .	
\end{align}
Combining \eqref{V2dot} and \eqref{robust_Vdotfin} with applying comparison principle, one can derive that there exist positive constants $M_1>0$ and $d_1>0$ such that the following decay of the norm holds 
\begin{align} 
V(t) + V_2(t) \leq M_1 (V(0) + V_2(0)) e^{-d_1 t}. 
\end{align}
Using the procedure in Section \ref{sec:allstability}, we conclude Theorem \ref{thm-robust}. 
\end{proof}

\section{Numerical Simulation} \label{sec:simulation} 
\begin{table}[t]
\vspace{2mm}
	
\caption{Physical properties of zinc}
\begin{center}
    \begin{tabular}{| l | l | l | }
    \hline
    $\textbf{Description}$ & $\textbf{Symbol}$ & $\textbf{Value}$ \\ \hline
    Liquid density & $\rho_{l}$ & 6570 ${\rm kg}\cdot {\rm m}^{-3}$\\ 
    Solid density & $\rho_{s}$ & 6890 ${\rm kg}\cdot {\rm m}^{-3}$\\ 
    Liquid heat capacity & $c_{l}$ & 390 ${\rm J} \cdot {\rm kg}^{-1}\cdot {\rm K}^{-1}$  \\  
    Solid heat capacity & $c_{s}$ & 390 ${\rm J} \cdot {\rm kg}^{-1}\cdot {\rm K}^{-1}$  \\  
    Liquid thermal conductivity & $k_{{\rm l}}$ & 130 ${\rm W}\cdot {\rm m}^{-1}$  \\ 
    Solid thermal conductivity & $k_{{\rm s}}$ & 100 ${\rm W}\cdot {\rm m}^{-1}$  \\ 
    Melting temperature & $T_{m}$ & 420 $^\circ{\rm C}$  \\ 
    Latent heat of fusion & $\Delta H^*$ & 120,000 ${\rm J}\cdot {\rm kg}^{-1}$ \\ \hline 
        \end{tabular}
\end{center}
\end{table}

Simulation results are performed by considering a strip of zinc whose physical properties of both liquid and solid phases are given in Table 1. The numerical model of the two-phase Stefan problem is obtained by boundary immobilization method combined with finite difference semi-discretization following \cite{kutluay97}. Under the identical choice of the physical parameters in the plant and control, we compare the performance of the proposed design in \eqref{Fullcontrol} (hereafter ``two-phase design") and the following design
\begin{align} \label{op-design}
q_{{\rm c}}(t) = - c \left(	\frac{k_{\rm l}}{\alpha_{{\rm l}}} \int_{0}^{s(t)} (T_{{\rm l}}(x,t) - T_{{\rm m}}) {\rm d}x + \gamma (s(t) - s_{{\rm r}}) \right),  
\end{align}
which is developed in our previous work in \cite{Shumon19journal} for the one-phase Stefan problem (hereafter ``one-phase design"). The stability under the ``one-phase design" is guaranteed by Theorem \ref{thm-robust} for the robustness analysis of the closed-loop system under the restriction of the setpoint to satisfy $q_{{\rm c}}(0) \geq 0$ for \eqref{op-design}. 

The material's length, the initial interface position, and the setpoint position are chosen as $L$ = 1.0 m, $s_0$ = 0.4 m, and $s_{{\mathrm r}}$ = 0.5 m. The initial temperature profiles are set as $T_{{\rm l},0}(x)= \bar{T}_{\rm l,0}(1-x/s_0) + T_{{\mathrm m}}$ and $T_{{\rm s},0}(x) = \bar{T}_{\rm s,0}(1-(L-x)/(L-s_0)) + T_{{\mathrm m}}$ with $ \bar{T}_{\rm l,0}$ = 10 $^\circ$C and $\bar{T}_{\rm s,0}$ = --200 $^\circ$C, of which the schematic is shown in Fig. \ref{fig:initial}. Then, the setpoint restrictions for both ``two-phase design" and ``one-phase design" are satisfied. The control gain is set as $c =$ 1.0 $\times$ 10$^{-2}$/s.  

\begin{figure}[t]
\centering 
{\includegraphics[width=3.0in]{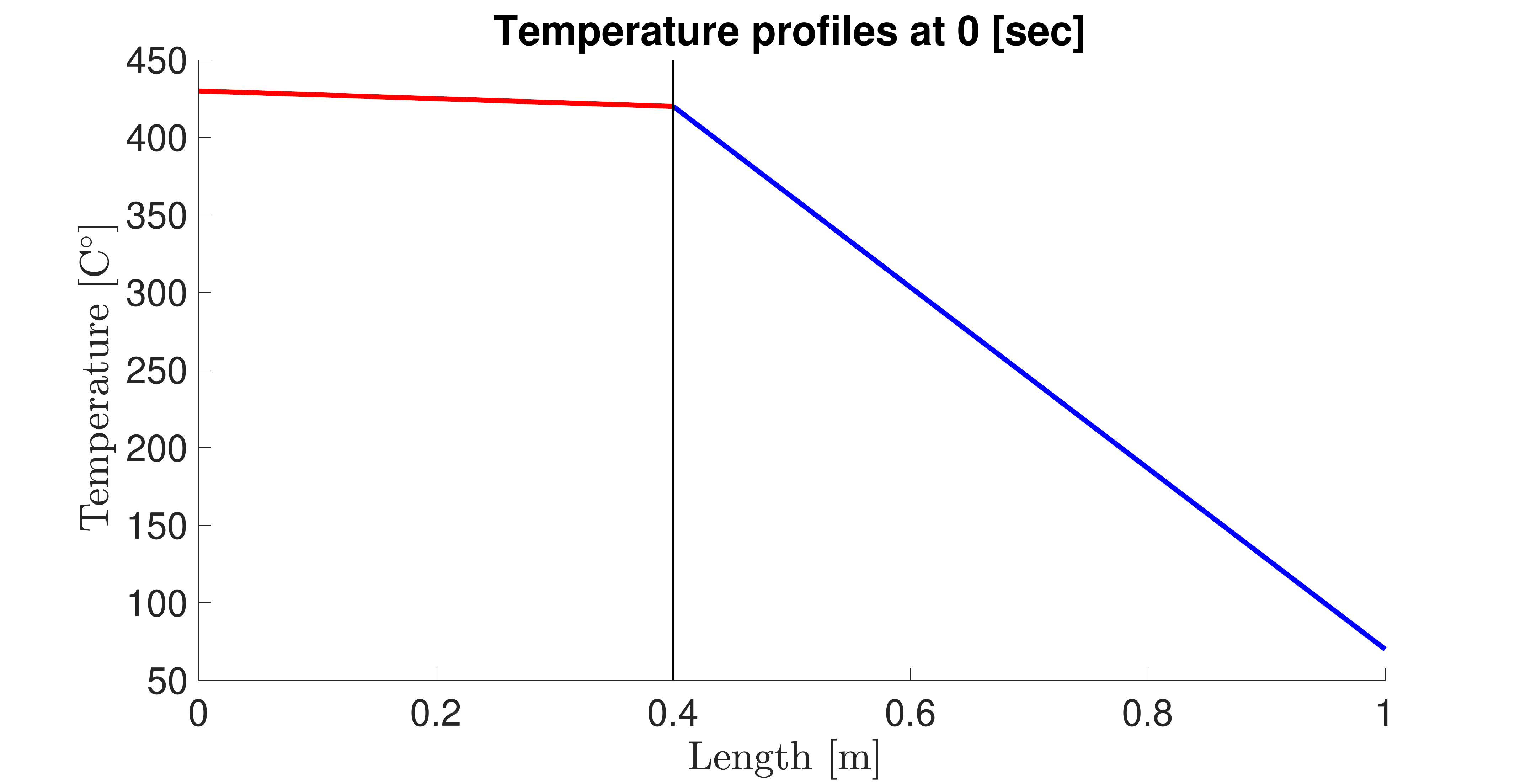}}\\
\caption{The initial temperature profiles on both liquid (red) and solid (blue).} \label{fig:initial}
\end{figure} 

\begin{figure}  
\centering 
\subfloat[Convergence of the interface to the setpoint $s_{{\rm r}}$ is observed for both controls, however, the proposed two-phase design achieves faster convergence as seen in the settling time in Fig. \ref{settling}. ]
{\includegraphics[width=3.0in]{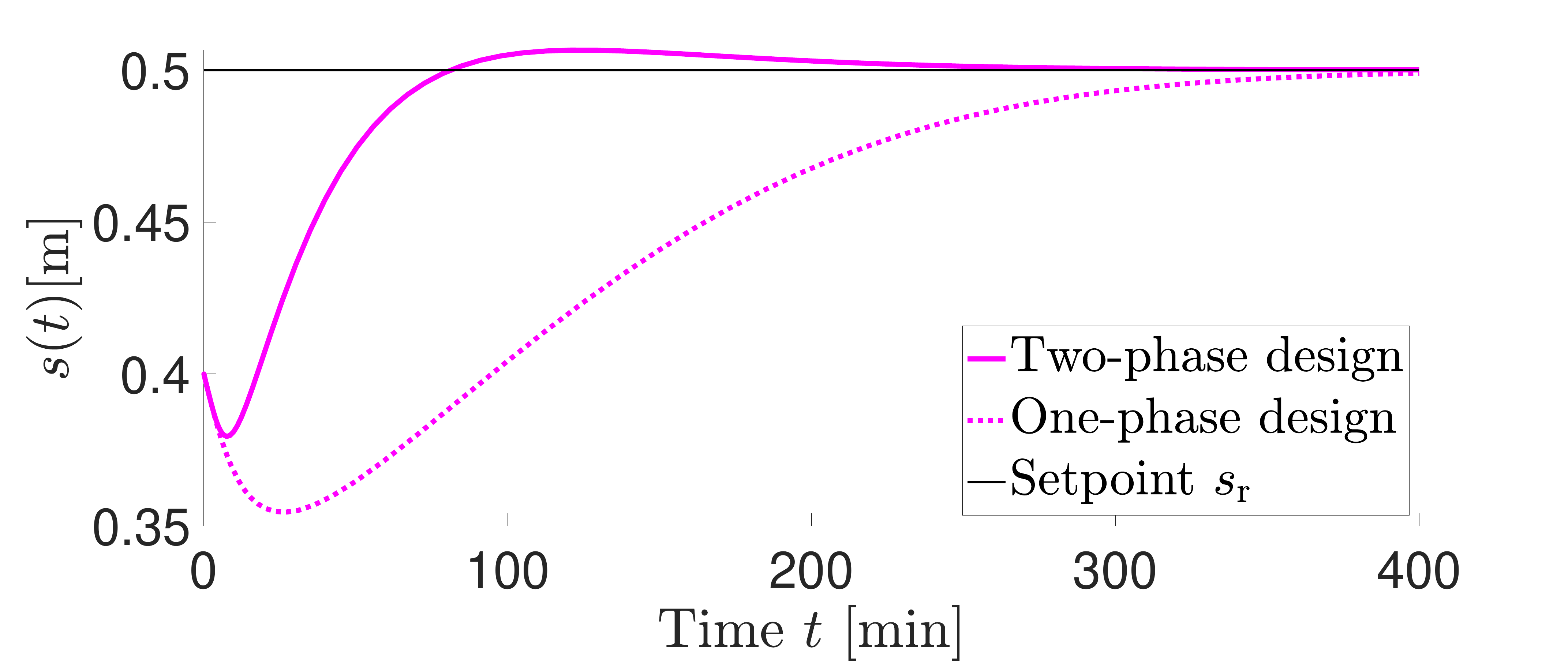}\label{fig:interface}}\\
\subfloat[Positivity of the heat input is satisfied for both controls.]
{\includegraphics[width=3.0in]{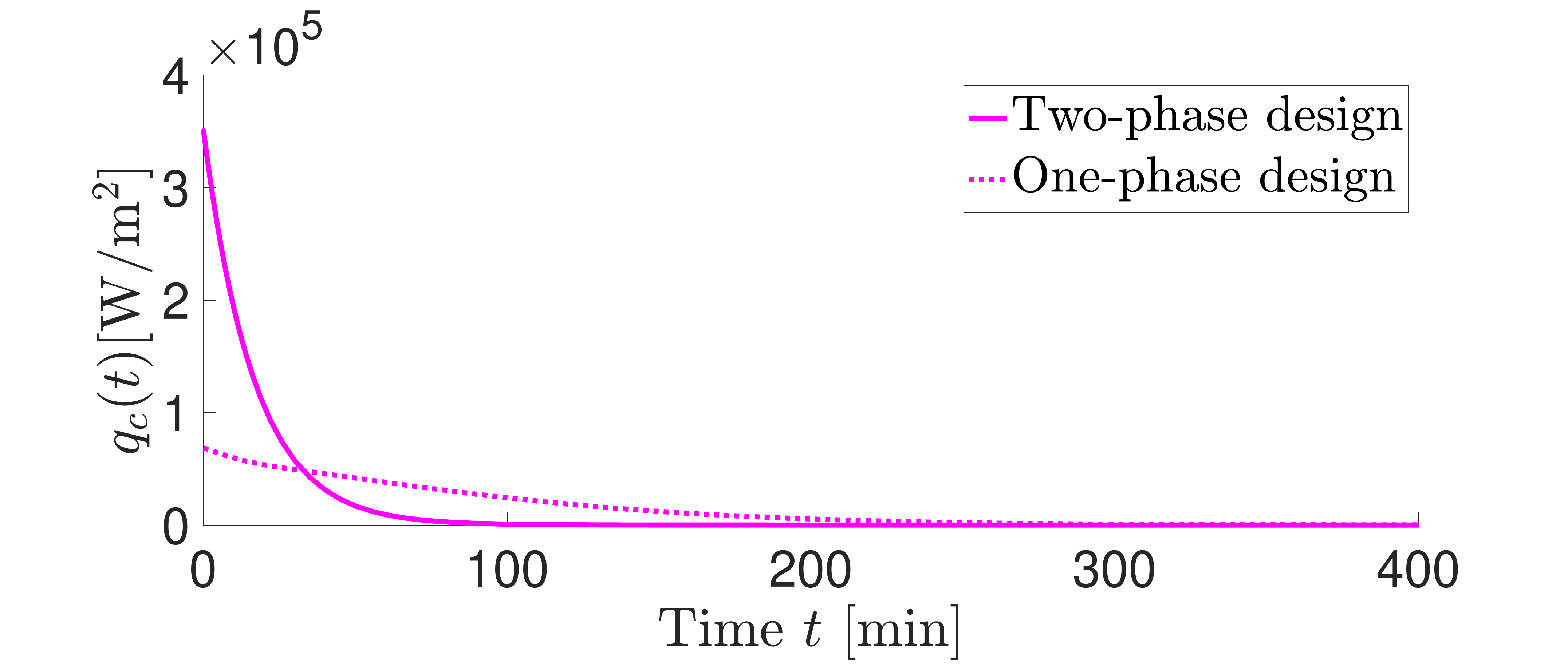}\label{fig:qc}}\\
\subfloat[The boundary temperature maintains above the melting temperature, and hence there is no appearance of a new solid phase from the controlled boundary $x=0$ in the liquid phase.]
{\includegraphics[width=3.0in]{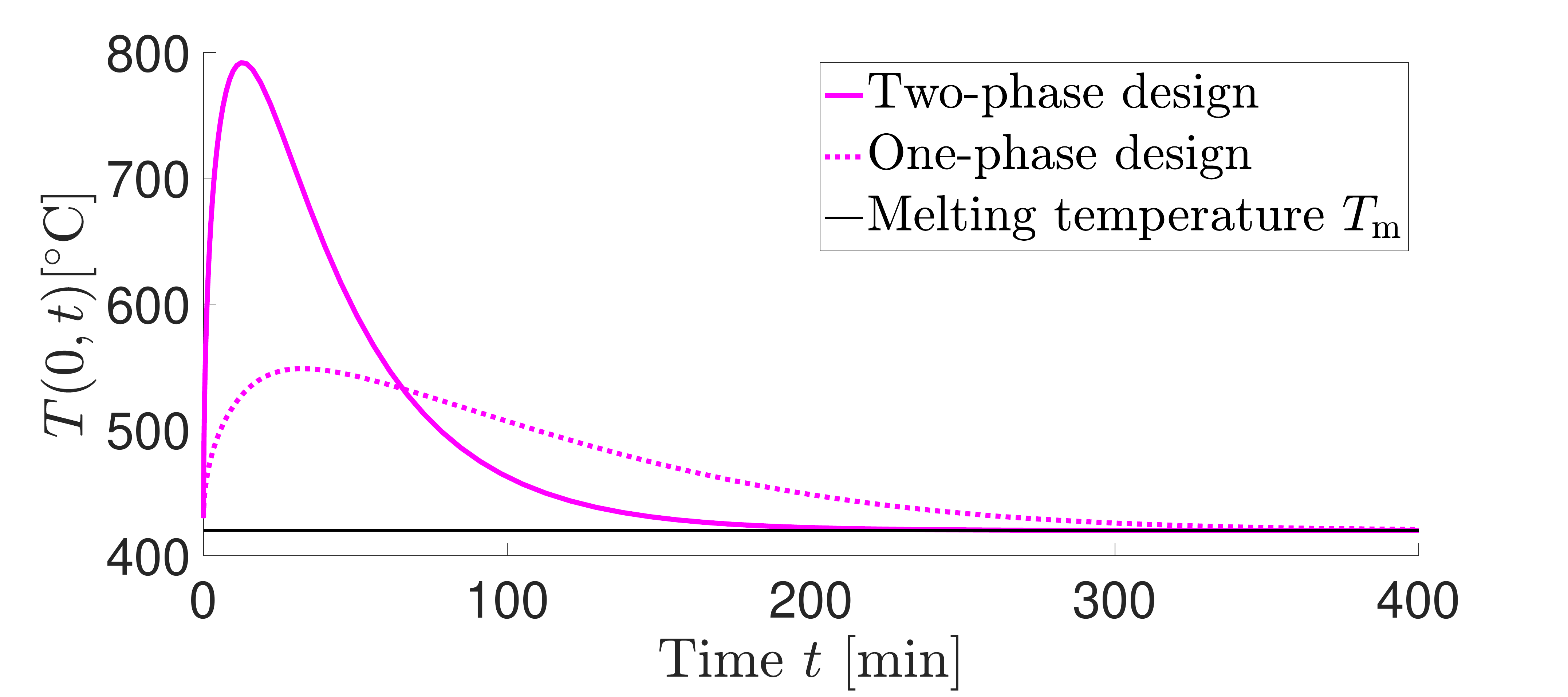}\label{fig:temp}}
\caption{The closed-loop responses under the proposed ``two-phase" design (pink solid) and the ``one-phase" design (pink dash).} 
\label{fig:main} 
\vspace{5mm} 
{\includegraphics[width=3.0in]{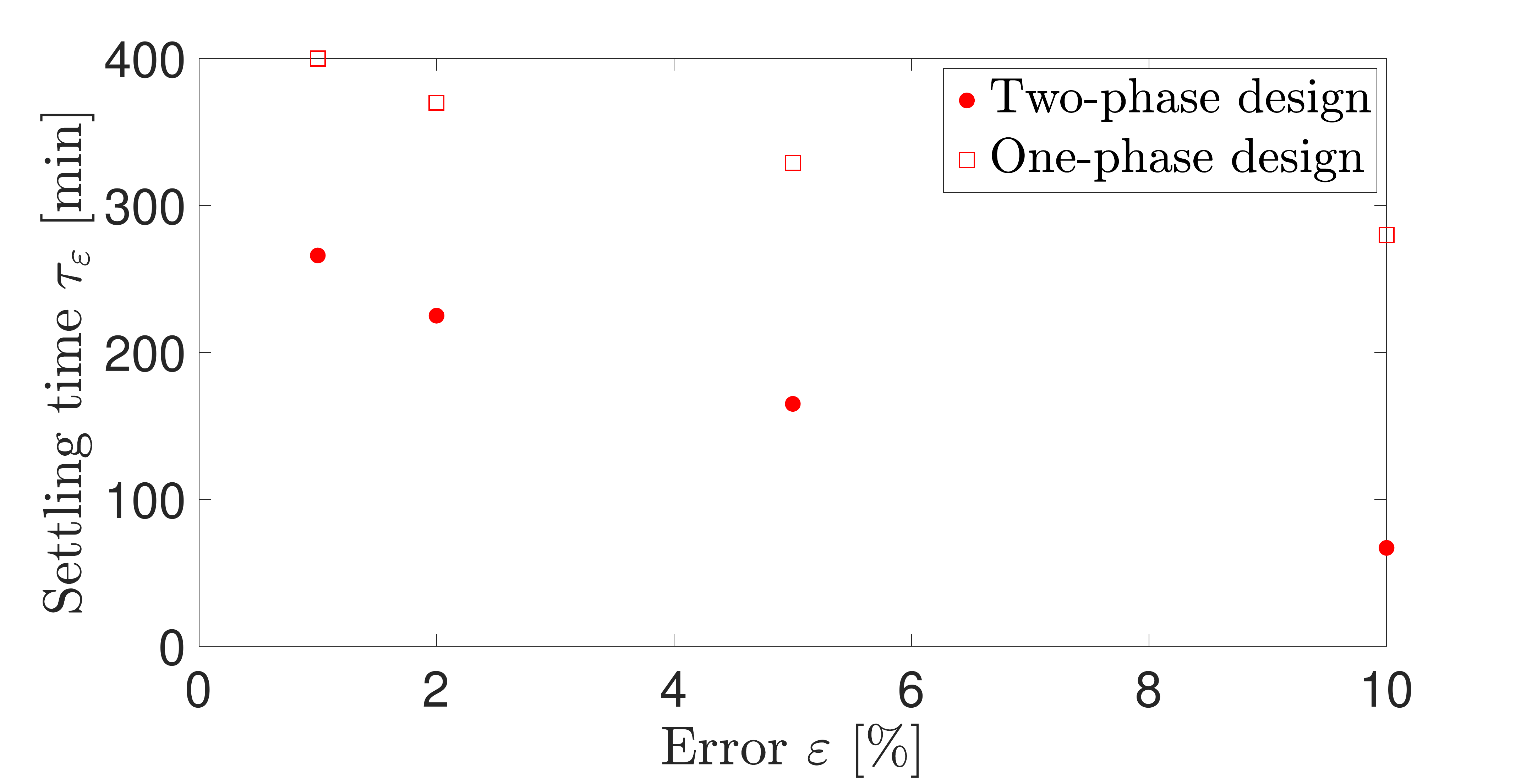}}
\caption{Settling time of the interface convergence in Fig. \ref{fig:main} (a). }
\label{settling} 
\end{figure}

The closed-loop responses are implemented as depicted in Fig \ref{fig:interface}-\ref{fig:temp} for both ``two-phase design" (solid) and ``one-phase design" (dash). Fig \ref{fig:interface} shows the dynamics of the interface $s(t)$. We can observe that $s(t)$ decreases at first due to the freezing caused by the initial temperature of the solid phase, and after some time the interface position increases and converges to the setpoint owing to the melting heat input. Moreover, the interface dynamics under the ``two-phase design" achieves faster convergence than that under the ``one-phase design" with having a little overshoot as seen from Fig. \ref{fig:interface}. Fig \ref{fig:qc} shows the dynamics of the closed-loop control, and Fig \ref{fig:temp} shows the dynamics of the boundary temperature of the liquid phase $T_{{\rm l}}(0,t)$. Fig \ref{fig:qc} illustrates the positivity of the heat input $q_{{\rm c}}(t)>0$, and Fig. \ref{fig:temp} illustrates the liquid boundary temperature being greater than the melting temperature, which are consistent with Lemma \ref{lem:closed-valid}. Hence, we can observe that the simulation results are consistent with the theoretical result we prove as model validity conditions and the stability analysis. 

To compare the performance on the convergence speed between ``two-phase design" and ``one-phase design," we investigate the settling time $\tau_{\ep}$ with respect to the error $\ep$ [\%] of the interface position relative to the setpoint, mathematically defined by 
\begin{align} 
\tau_{\ep} := \inf_{\tau \geq 0} \left\{ \tau \bigg| |s(t) - s_{{\rm r}}| \leq |s_0 - s_{{\rm r}}| \frac{\ep}{100}, \quad \forall t \geq \tau \right\} . 
\end{align} 
Fig. \ref{settling} shows the value of $\tau_{\ep}$ with $\ep = $ 10, 5, 2, 1 [\%]. From the figure, it is observed that the convergence speed of ``two-phase design" compared to the speed of the ``one-phase design" is approximately four times faster for $\ep$ = 10[\%], two times faster for $\ep$ = 5 [\%], one and half times faster for both $\ep$ = 2 [\%] and 1 [\%], respectively. Hence, Fig. \ref{settling} validates superior performance of the proposed ``two-phase design" compared to the ``one-phase design" developed in our previous work \cite{Shumon19journal}.

\section{Conclusion and Future work} \label{sec:conclusion} 
In this paper we presented the full state feedback control law of a single heat boundary input for the two-phase Stefan problem to stabilize the moving interface position at a desired setpoint. The main contribution of the paper is that we theoretically prove the global exponential stability of the closed-loop system of the two-phase Stefan problem with designing the state feedback control law by employing energy shaping and backstepping. While our present result is only on the stabilization of the moving interface at the setpoint with restricting the equillibrium temperature to only the uniform melting temperature, the simultaneous stabilization of the interface position and the temperature profile at arbitrary setpoint and temperature profiles linear in space following recent results in \cite{yu19_shock} for traffic congestion control with moving shockwave is considered as our future work. The application of extremum seeking control for online optimization of static maps to the Stefan problem following the recent results of \cite{Feiling18} is also a potential direction. 

\section*{Acknowledgement} 
The authors gratefully acknowledge the funding support by National Science Foundation (NSF Award Number:1562366). 

\section*{References}

\bibliography{mybibfile}

\end{document}